\pdfoutput=1 
\pdfoutput=1 
\pdfoutput=1 
\pdfoutput=1 
\pdfoutput=1 
\documentclass[journal]{IEEEtran}

\usepackage{array}
\usepackage{upgreek}      
\usepackage{amsmath}
\usepackage{amssymb}
\usepackage{arydshln}       
\usepackage{amsthm}
\usepackage{pbox}
\usepackage{verbatim}
\usepackage{color}
\usepackage{enumerate}
\usepackage[utf8]{inputenc}
\usepackage[english]{babel}
\usepackage{cite}
\usepackage{algorithm}
\usepackage{algorithmic}
\usepackage{empheq}
\usepackage{xpatch}
\usepackage[titles]{tocloft}
\usepackage{fmtcount}
\usepackage{subcaption}
\usepackage{tabularx}
\usepackage{makecell}

\usepackage{float}
\usepackage[final]{graphicx}
\graphicspath{{./figure/}}

\theoremstyle{theorem}
\newtheorem{Thm}{Theorem}[section]

\newtheorem{Lem}[Thm]{Lemma}

\newtheorem{Def}[Thm]{Definition}

\newtheorem{Rem}[Thm]{Remark}

\theoremstyle{remark}

\DeclareMathAlphabet{\mathcal}{OMS}{cmsy}{m}{n}

\usepackage{soul}

\usepackage{color} 
\usepackage[colorinlistoftodos, textwidth=3cm, shadow]{todonotes}

\newcommand{\bm}[1]{\boldsymbol{#1}}

\newcommand{\RNum}[1]{\uppercase\expandafter{\romannumeral #1\relax}}

\usepackage[hidelinks]{hyperref}
\usepackage[noabbrev, capitalise]{cleveref}

\usepackage{doi}

\begin{document}
	\title{False Data Injection Attacks on Hybrid AC/HVDC \\
		Interconnected Systems with Virtual Inertia -- \\ Vulnerability, Impact and Detection}
	
	\author{Kaikai~Pan$^{*}$,
		    Elyas~Rakhshani,
	        and~Peter~Palensky
	\thanks{The authors are with the Delft University of Technology, The Netherlands. ($^{*}$ Correspondence: \tt{kaikaipan15}@gmail.com)} 
	}
	\maketitle 

\begin{abstract}
	Power systems are moving towards hybrid AC/DC grids with the integration of HVDC links, renewable resources and energy storage modules. New models of frequency control have to consider the complex interactions between these components. Meanwhile, more attention should be paid to cyber security concerns as these control strategies highly depend on data communications which may be exposed to cyber attacks. In this regard, this article aims to analyze the false data injection (FDI) attacks on the AC/DC interconnected system with virtual inertia and develop advanced diagnosis tools to reveal their occurrence. We build an optimization-based framework for the purpose of vulnerability and attack impact analysis. Considering the attack impact on the system frequency stability, it is shown that the hybrid grid with parallel AC/DC links and emulated inertia is more vulnerable to the FDI attacks, compared with the one without virtual inertia and the normal AC system. We then propose a detection approach to detect and isolate each FDI intrusion with a sufficient fast response, and even recover the attack value. In addition to theoretical results, the effectiveness of the proposed methods is validated through simulations on the two-area AC/DC interconnected system with virtual inertia emulation capabilities. 
\end{abstract}

\begin{IEEEkeywords}
	AC/DC system, virtual inertia emulation, frequency control, false data injection attacks
\end{IEEEkeywords}

\IEEEpeerreviewmaketitle

\section{Introduction}\label{sec:intro}

\IEEEPARstart{R}{ecent} advances in power electronics have made HVDC links and renewable energy resources (RES) more popular in power system applications \cite{Rakhshani2016a}. To better support the frequency control in the modern scenarios, several efforts have been carried out to develop different approaches for inertia emulation tasks \cite{Xu2018,Dhingra2018}. Besides, this transformation does not only lead to more adaptation of conventional control schemes such as the primary and secondary frequency control to handle complex interactions between these components, but also introduce an increasing dependence on data communications. The new model of frequency control must consider hybrid AC/HVDC multi-area systems incorporating virtual inertia emulators \cite{Mosca2019}. Besides, such control process would also rely heavily on communication networks and numerous cyber devices to achieve reliable and in-time data exchange. However, the corresponding communication channels are lack of sufficient cyber security mechanisms, and the wide communication surface of the hybrid AC/DC grids makes them more exposed to cyber intrusions, including false data injection (FDI) attacks \cite{Liang2017a, Sridhar2010}. Notably, as we know, the DC grid has a low tolerance to a fault; if a single fault can mislead a fast response, so as an attack. Indeed, a deliberate attacker can result in severe consequences on system frequency stability \cite{Tan2017}. Thus one need advanced tools to understand the attack strategies and detect them sufficiently fast in the context of AC/DC interconnected system which is also facilitated by inertia emulators. 

\subsection{Related Work}
Nowadays, the matter of virtual inertia is one of the hottest topics in the power system 
studies \cite{Fang2019,Ruttledge2015}. Different methods for inertia emulation have been proposed for the integration of RES into the power grid by using Energy Storage Systems (ESS) \cite{Datta2013}. In the AC/DC interconnected system, the dynamic effects of these components in the inertia emulation process 
need to be reflected in the high-level frequency control loop, which also highlights the importance of cyber-secure data communications \cite{Dudurych2017}. We can see that many research activities have been done on the cyber security issues of pure AC power systems. Vulnerability and impact analysis of the AC grid to cyber attacks can be found in \cite{TeixeiraSou2015,Esfahani2010}. For the detection of FDI attacks which corrupt the frequency stability, model-based detectors have emerged mainly from observer-based approaches in the framework of differential-algebraic equations (DAEs) \cite{Nyberg2006a}. For instance, in \cite{Ameli2018a} an unknown input observer was employed to detect FDI attacks on the conventional multi-area AC systems. To be noted, these observers for attack detection usually have the same degree as the system dynamics, which can cause troubles in the online implementation especially for large-scale systems \cite{Pan2020}. Other FDI attack detectors may come from statistical methods which always have prior assumptions on the distribution of measurement errors \cite{Sridhar2014}. 

For the hybrid AC/DC grid with virtual inertia, however, there are still very few studies that have focused on its cyber security research \cite{Pan2018}. To do that, it requires extensive knowledge on the dynamics of the complex AC/DC system with emulated inertia and also methods or algorithms for cyber security analysis. The work in \cite{Amir2019} studied the effect of attacks on the HVDC system and impact on the dynamic voltage stability, and proposed a predictive control based method to detect these attacks. The authors in \cite{Brown2018} demonstrated mechanisms by which an attacker could cause system-wide unstable oscillations via loads with emulated inertia control. 
Recently, the literature \cite{Roy2019} studied the possible impact of FDI attacks on the secondary frequency control of low inertia grid under deregulated power systems, and proposed a detection mechanism based on the load forecasts, by ensuring the availability and accuracy of such additional data. To conclude, for vulnerability and impact analysis, more efforts are still needed to learn the optimal FDI attack strategies especially in the context of AC/DC interconnected multi-area system with emulated inertia. Besides, a low-order diagnosis tool for a fast response in the presence of FDI attacks is more preferred in practice. To the best of our knowledge, the following question is still not answered sufficiently, 
\begin{flushleft}
	\centering
	{\emph{Would it be possible to quantify the vulnerability and impact of the 
			FDI attacks on the AC/DC interconnected system with emulated inertia, and propose a diagnosis tool to detect such malicious intrusions in time?}}
\end{flushleft} 

\subsection{Contributions and Outline}
In this article we aim to address the question above. For this purpose, we would first build the frequency dynamics model for the AC/DC interconnected system with virtual inertia emulation capabilities. A high-level control structure is presented, which also illustrates possible vulnerable points of the system to FDI attacks. Next, we introduce impact indices to evaluate the effects of FDI attacks on the frequency stability. We formulate an optimization-based framework to assess the vulnerability of the AC/DC interconnected system to the FDI attacks. In the end, a detector in the form of residual generator with adjustable design parameters is proposed to detect, isolate and even recover each FDI intrusion. The main contributions of this article are: 
\begin{itemize}
	\item[(i)] Unlike many existing literature, we go beyond the study on the normal AC system 
	that we explore the vulnerability and impact of the AC/DC interconnected system with emulated inertia to the FDI attacks. A well-constructed optimization-based framework is built to analyze the optimal attack strategies on achieving undetectability and desired impact. From both the theoretical analysis and numerical results, we have pointed out that the hybrid AC/DC system with virtual inertial can be more vulnerable to FDI attacks, compared with the one without virtual inertia and the conventional AC system.  
	\item[(ii)] We further explore the attack detection problem in the AC/DC interconnected system with virtual inertia. Different from observer-based approaches or other prediction-based techniques, we formulate a type of residual generator that can have adjustable design parameters for fast responses in attack detection. It is guaranteed that the resulting residual signal is decoupled from the system states (e.g., frequency dynamics) and normal load disturbances, and can recover the attack magnitude in the steady-state value of the residual. We also provide attack isolation method together with conditions of attack detectability and isolability. 
\end{itemize}

In Section~\ref{sec:sys_model}, we describe the frequency dynamics model of the AC/DC interconnected system with emulated inertia, and vulnerable points are illustrated in the high-level control structure. 
Section~\ref{sec:fdia} introduces the FDI attack and its optimal strategy to be disruptive and stealthy. Besides, we propose an optimization-based framework to analyze the vulnerability and impact of the system to such attacks. The FDI attack detection approach is developed in Section~\ref{sec:detect} where we also show its capabilities of attack isolation and recovery. Finally, numerical results are 
reported in Section~\ref{sec:result}. 

\begin{figure*}[t!p]
	\centering
	\includegraphics[scale=0.58]{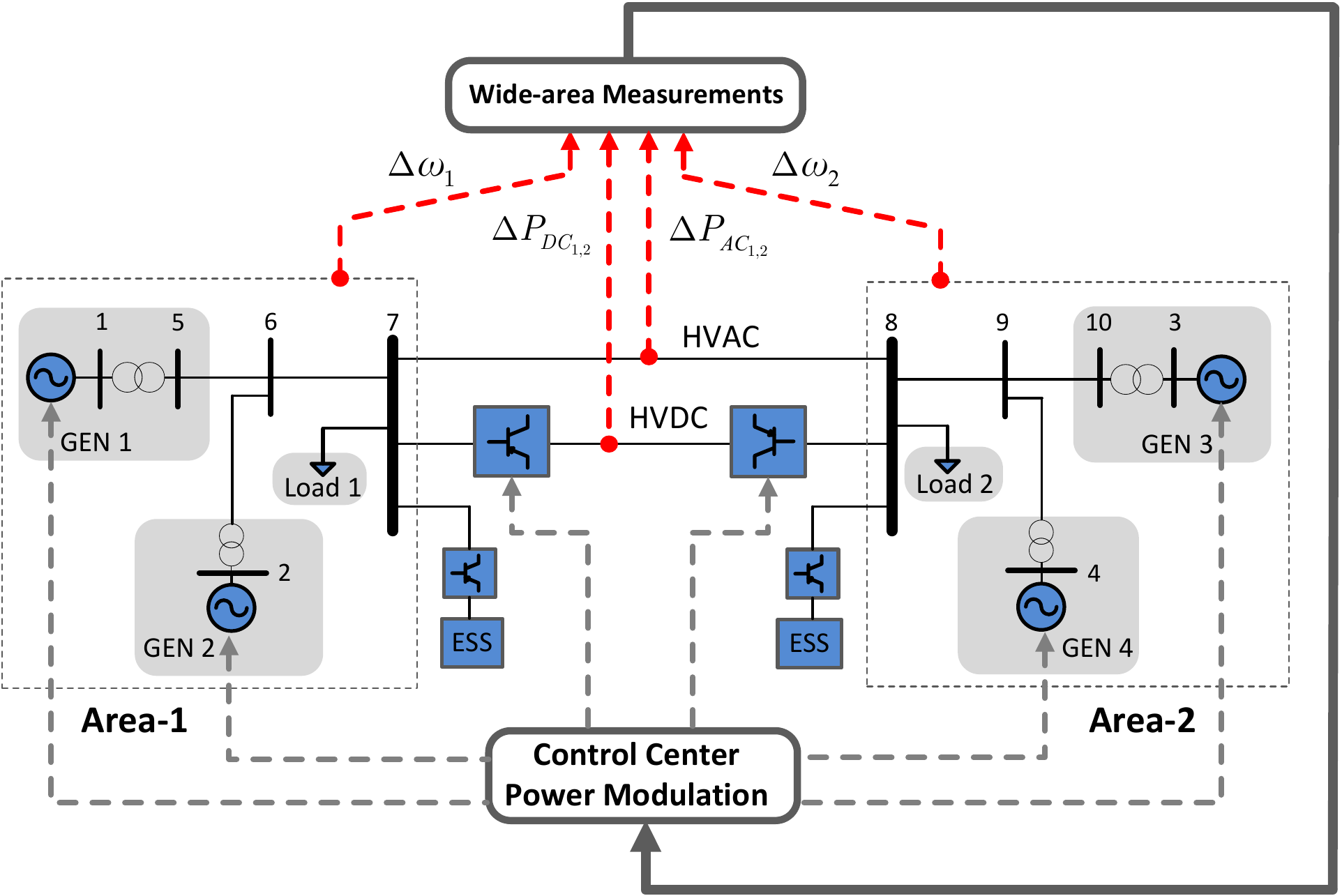}
	\caption{The configuration of the two-area 
		system with parallel AC/DC links and bulk ESS. Since the communication network between the transducers and the control center is very vulnerable, the FDI attacks of this article are mainly launched on the uploading channels (the red lines) of wide-area frequency and AC/DC power flow measurements \cite{Roy2019}.} 
	\label{fig:acdcsys}
\end{figure*}
%

\section{Hybrid AC/DC System Modeling}\label{sec:sys_model}

In this section, the frequency dynamics model in a high-level control structure is presented to simulate and analyze the behaviour of the hybrid AC/DC interconnected system with inertia emulation capabilities.

\subsection{The concept of virtual inertia} \label{subsec:inertia}

For the inertia emulation task, one effective way is to use the derivation of the system frequency proportionally to adjust the active power reference of a converter. Then the emulated inertia can contribute to improve the performance of the system dynamics on the inertia response. 
This control concept is the derivative control which calculates the rate of change of frequency (ROCOF), 
and can be described as %
\begin{align}\label{eq:virt_inert}
\Delta P_{emu} = k_{a} {\omega}_{o} \Delta \dot{\omega} \, ,  
\end{align}
where $\Delta P_{emu}$ is the power reference, $\omega_{o}$ is the nominal frequency and $\Delta \omega$ is the frequency deviation, and $k_{a}$ denotes the inertial proportional conversion gain which can be selected according to an iterating tuning approach for minimizing the frequency deviations; we refer to \cite{Rakhshani2016b} for more details.

\subsection{AC/DC interconnected system with emulated inertia} 
\label{subsec:modeling}

The diagram of an exemplary two-area system with parallel AC/DC links and added bulk ESS for providing virtual inertia is presented in Figure~\ref{fig:acdcsys}. In this test system, two converters are added as interfaces for controlling the behavior of the ESS modules in accordance with HVDC coordination and secondary frequency control signals to reduce the deviations of system frequency during contingencies. 
The wide-area frequency and AC/DC power flow measurements are collected in each area and tie-line, and sent to the control center via communication networks for the HVDC coordination and 
secondary frequency control operation. Note that instead of these external measurements, 
the virtual inertial emulation is using the local information only, which implies a faster response.
The external measurements are transmitted through dedicated communication networks whose protocols (e.g., DNP3, IEC61850) are usually not equipped with adequate cyber security features \cite{Pan2017a}. Considering the vulnerabilities within these communication channels, attack scenarios of this article mainly focus on the uploading paths of wide-area measurements, as depicted in Figure~\ref{fig:acdcsys}.

Next, we describe the frequency dynamics model of the two-area AC/DC interconnected system with 
ESS for inertia emulation. It should be noted that in this article a high-level control structure is adopted to reflect the primary and secondary frequency control properties based on \cite{Rakhshani2017a}. This control structure is illustrated in Figure~\ref{fig:acdcess_contrl}. Thus as we can see from Figure~\ref{fig:acdcess_contrl},  
the frequency deviation of Area $i$ in the Laplace domain can be expressed as 
\begin{equation} \label{eq:freq_areai}
\begin{aligned}
\Delta \omega_{i}(s) \, = \, & \frac{K_{p_{i}}}{1+sT_{p_{i}}} \big[ \Delta P_{m_{i}} -\Delta P_{L_{i}} - (\Delta P_{AC_{i,j}}  \\
& + \Delta P_{DC_{i,j}} + \Delta P_{ESS_{i}}) \big]  \, ,  
\end{aligned}
\end{equation}
where $\Delta P_{ESS_{i}}$ and $\Delta P_{m_{i}}$ are the emulated power from ESS and the total active power from all generation units (GENs) within Area $i$, i.e., $\Delta {P}_{m_{i}} = \sum_{g=1}^{G_{i}} \Delta {P}_{m_{i,g}}$ where $G_{i}$ denotes the number of participated GENs. 
$K_{p_{i}}$ is the system gain and $T_{p_{i}}$ is the system time constant. The total load variation is mentioned by $\Delta P_{L_{i}}$ in Area $i$, while $\Delta P_{AC_{i,j}}$ and $\Delta P_{DC_{i,j}}$ are the AC and DC power flow deviations between Area $i$ and Area $j$, respectively. We can further have
\begin{align}\label{eq:dpmi}
\Delta {P}_{m_{i,g}}(s) = \frac{1}{1+ sT_{ch_{i,g}}} \big[\frac{\Delta \omega_{i}}{R_{i,g} \times 2\pi} - \phi_{i,g}\Delta P_{agc_{i}} \big] \, , 
\end{align}
\begin{align}\label{eq:dptieacij}
\Delta {P}_{AC_{i,j}}(s) = \frac{T_{ij}}{s} \big[\Delta \omega_{i} - \Delta \omega_{j} \big] \, .  
\end{align}
In \eqref{eq:dpmi} and \eqref{eq:dptieacij}, $R_{i,g}$ denotes the droop of each generation unit, $T_{ch_{i,g}}$ is the overall time constant of the turbine-governor model, $\Delta P_{agc_{i}}$ is the output signal from the secondary control of Area $i$ for power reference of each generation unit, $\phi_{i,g}$ is an area participating factor, and $T_{ij}$ denotes the power coefficient of the AC line between Area $i$ and Area $j$.

To model the HVDC link, we use the concept of Supplementary Power Modulation Controller (SPMC). As illustrated in Figure~\ref{fig:acdcess_contrl}, the SPMC can be designed as a high-level damping controller whose inputs consist of deviations of frequencies in the interconnected areas and the deviation of the transmitted power in the AC line. Then the output of SPMC is used for the HVDC link to generate the desired DC power. Thus the coordinated control strategy acting as a supplementary power modulation for this DC link can be described by 
\begin{align}\label{eq:xdc}
\Delta P_{DC_{ref}} = K_{i}\Delta \omega_{i} + K_{j}\Delta \omega_{j} + K_{AC} \Delta {P}_{AC_{i,j}} \, ,  
\end{align}
\begin{align}\label{eq:dpdcij}
\Delta {P}_{DC_{i,j}}(s) = & \frac{1}{1 + s T_{DC}} \Delta P_{DC_{ref}} \, ,  
\end{align}
where $\Delta P_{DC_{ref}}$ denotes the desired DC power reference 
for the HVDC line , $T_{DC}$ is the time constant of the DC link, and $K_{i}$, $K_{j}$ and $K_{AC}$ are used as control gains.

For the inertia emulation process, according to the control law 
in Section~\ref{subsec:inertia}, the deviation of active power output from the ESS in each area, namely $\Delta P_{ESS_{i}}$, can be written as 
\begin{align}\label{eq:dpessi}
\Delta {P}_{ESS_{i}}(s) = \frac{J_{em_{i}}}{1+ sT_{ESS_{i}}} [s \Delta \omega_{i}(s)] \, ,  
\end{align}
where $T_{ESS_{i}}$ is the time constant of the derivative-based components. Notably, in practice, the derivative based control strategy might be sensitive to the noises especially during the measurements of frequency signals. Therefore, a low-pass filter can be added to the model for eliminating the effects of noises. In this study, the dynamics of such a filter with storage elements is considered by the time constant $T_{ESS_{i}}$. From \eqref{eq:dpessi}, there will be two gains ($J_{em_1}$and $J_{em_2}$) representing the virtual inertia for these two areas. 
As stated in Section~\ref{subsec:inertia}, the values of these two gains can be obtained from the iterating optimization approach in \cite{Rakhshani2016b}. For the inertia emulation control 
of this article, it is assumed that we have enough energy stored in the DC side of the converter station, and the stored energy is only used for a short period of time ($2 \, \mathrm{s}$ to $5 \, \mathrm{s}$) to provide virtual inertia. Besides, we have simplified the ESS model in this article that only its dynamic effects on the high-level frequency control has been considered. The detail of state of the charge is not modeled but its value is assumed to be always brought back close to the higher limit during off-peak period.

\begin{figure}[t!p]
	\centering
	\includegraphics[scale=0.67]{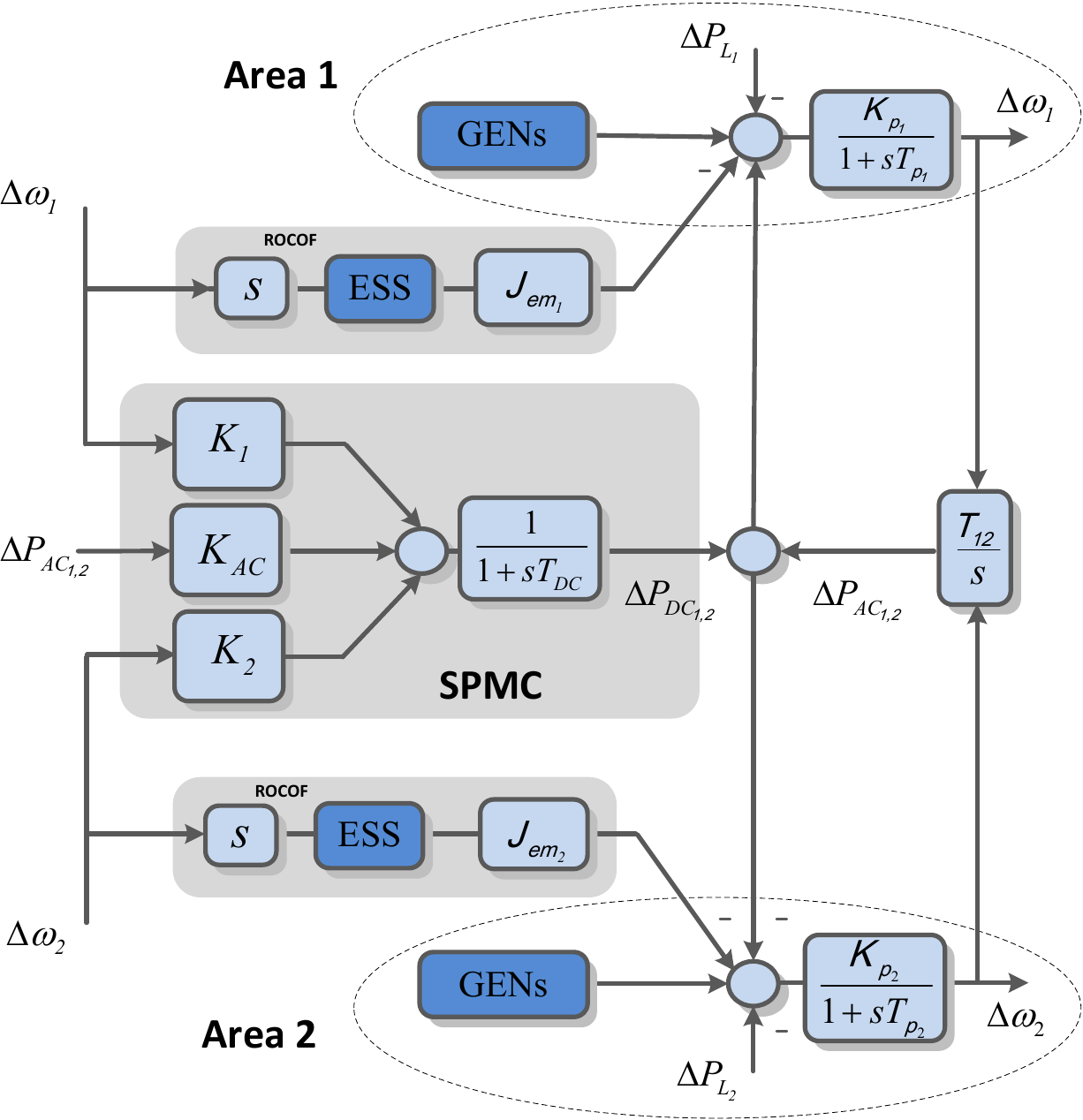}
	\caption{A high-level control structure for the frequency dynamics of the two-area AC/DC interconnected system with ESS for inertia emulation.} \label{fig:acdcess_contrl}
\end{figure}

The part of secondary frequency control calculates the Area Control Error (ACE) of Area $i$. The ACE signal now contains the deviations of frequency in that area and both AC/DC power flows, 
and acts as an input for the integral control action, 
\begin{align}\label{eq:acei}
ACE_{i} = \frac{{\beta}_{i}}{2\pi} \Delta \omega_{i} + (\Delta P_{AC_{i,j}} + \Delta P_{DC_{i,j}} )  \, ,  
\end{align}
\begin{align}\label{eq:agci}
\Delta P_{agc_{i}} = K_{I_{i}}\frac{ACE_{i}}{s} \, ,   
\end{align}
where $\beta_{i}$ is the frequency bias and $K_{I_{i}}$ represents the integral gain of the 
secondary frequency controller. Thus finally, based on the equations \eqref{eq:freq_areai} to \eqref{eq:agci}, the frequency dynamics model of the high-level control for the exemplary two-area AC/DC interconnected system with emulated inertia can be presented as the following state equation, 
\begin{align}\label{eq:spx_twoarea}
\dot{\bm{X}}(t) \, = \, \bm{A}_{c} \bm{X}(t) \, + \, \bm{B}_{c,d} \bm{d}(t) \, . 
\end{align}
where $\bm{X} \in \mathbb{R}^{n_{X}}$ represents the vector of all state variables and $\bm{d} \in \mathbb{R}^{n_{d}}$ is the system input of load variations, namely, 
\begin{equation}\label{eq:spx_xd}
\begin{aligned}
\bm{X} := \Big[ & \Delta \omega_{1} \ \Delta \omega_{2} \ \Delta {P}_{m_{1,1}} \ \Delta {P}_{m_{1,2}} \ \Delta {P}_{m_{2,1}} \ \Delta {P}_{m_{2,2}} \ \Delta P_{agc_{1}} \ \\
& \ \ \Delta P_{agc_{2}} \ \Delta P_{AC_{1,2}} \ \Delta P_{DC_{1,2}} \ \Delta {P}_{ESS_{1}} \ \Delta {P}_{ESS_{2}} \Big]^{\top}, \\
& \quad \quad \quad \quad \quad \ \bm{d} = \Big[ \Delta P_{L_{1}} \ \, \Delta P_{L_{2}} \Big]^\top. \nonumber
\end{aligned}
\end{equation}
Besides, $\bm{A}_{c}$ and $\bm{B}_{c,d}$ are constant matrices. Overall, in the two-area AC/DC interconnected system which also has two inertia emulators, there are three new state variables ($\Delta P_{DC_{1,2}}$, $\Delta {P}_{ESS_{1}}$ and $\Delta {P}_{ESS_{2}}$) of synchronous controllers that would be added, compared with the normal AC system. In addition to \eqref{eq:spx_twoarea}, we can also derive an output model where the wide-area measurements of frequencies and AC/DC power flows are available in the control center of Figure~\ref{fig:acdcsys}. Thus we can have
\begin{align}\label{eq:spy_twoarea}
{\bm{Y}}(t) \, = \, \bm{C} \bm{X}(t) \, , 
\end{align}
where $\bm{Y} \in \mathbb{R}^{n_{Y}}$ represents the system output and $\bm{C}$ is the output matrix. 
For the purpose of numerical analysis, \eqref{eq:spx_twoarea} and \eqref{eq:spy_twoarea} need to be discretized. To obtain the analytical solution for the discretization, the matrices $\bm{A}$ and $\bm{B}_{d}$ of sampled discrete-time model with a sampling-period $T_{s}$ become \cite{Ogata1995},     
\begin{equation}\label{eq:zoh}
\bm{A} = e^{\bm{A}_{c}T_{s}} \, , \quad \bm{B}_{d} = \int_{t = 0}^{T_s} e^{\bm{A}_{c}(T_{s}-t)}\bm{B}_{c,d} \mathrm{d} t \, . \\
\end{equation}
%

\section{FDI Attacks on the Hybrid AC/DC System}\label{sec:fdia}

As discussed in the proceeding, we consider a high-level structure of control and security for frequency dynamics of the multi-area AC/DC system with virtual inertia.

\subsection{FDI attack basics: vulnerability and impact} \label{subsec:fdi_basic}

An FDI attack can modify the wide-area measurements to a lower or higher value. Thus the system output discrete-time model under FDI attacks can be described by 
\begin{align}\label{eq:spy_attacked}
\tilde{\bm{Y}}[k] \, = \, \bm{C}\bm{X}[k] \, + \, \bm{f}[k] \, ,
\end{align}
where $\tilde{\bm{Y}}[\cdot]$ is the corrupted output and $\bm{f}[\cdot]$ denotes FDI attacks. In this article, we mainly consider the general ``stationary'' FDI attack where it occurs as a constant bias injection~$\bm{f}$ at an unknown time instance $k_{\min}$. Other scenarios such as scaling, ramp, pulse and random attacks are referred to \cite{Sridhar2014}. These corruptions on the outputs would affect the dynamics of the controllers and consequently the involved system. As shown in Figure~\ref{fig:acdcsys}, the FDI attacks are mainly on the wide-area frequency and AC/DC power flow measurements, which would corrupt the HVDC coordination and the secondary frequency control. Recall that the inertia emulation is using local measurements only and thus not attacked. 
Hence for instance, an FDI attack on the AC power flow measurement between Areas $i$ and $j$, say $f_{AC_{i,j}}$, can manipulate both the SPMC control and the secondary frequency control loops. Thus this attack can change \eqref{eq:xdc} and \eqref{eq:acei} into the following equations respectively,
\begin{equation}\label{eq:xdc_attack}
\Delta {P}_{DC_{ref}} = K_{i}\Delta \omega_{i} + K_{j}\Delta \omega_{j} + K_{AC} (\Delta {P}_{AC_{i,j}} + f_{AC_{i,j}}) \, , \nonumber 
\end{equation}
\begin{equation}\label{eq:acei_attack}
ACE_{i} = \frac{{\beta}_{i}}{2\pi} \Delta \omega_{i} + (\Delta P_{AC_{i,j}} + f_{AC_{i,j}} + \Delta P_{DC_{i,j}} )  \, . \nonumber  
\end{equation}
Thus the state equation under FDI attacks in the discrete-time mode can be expressed as 
\begin{align}\label{eq:spx_twoarea_attack}
{\bm{X}}[k+1] \, = \, \bm{A} \bm{X}[k] \, + \, \bm{B}_{d} \bm{d}[k] + \, \bm{B}_{f} \bm{f}[k] \, , 
\end{align}
where the matrix $\bm{B}_{f}$ relates FDI attacks to the system states. Note $\bm{B}_{f}$ is obtained through the same matrix transformation as $\bm{B}_{d}$ in \eqref{eq:zoh}, while the matrix $\bm{B}_{c,f}$ in the continuous system model depends on the specific attack scenario. 

\begin{figure*}
	\centering
	\includegraphics[scale=0.77]{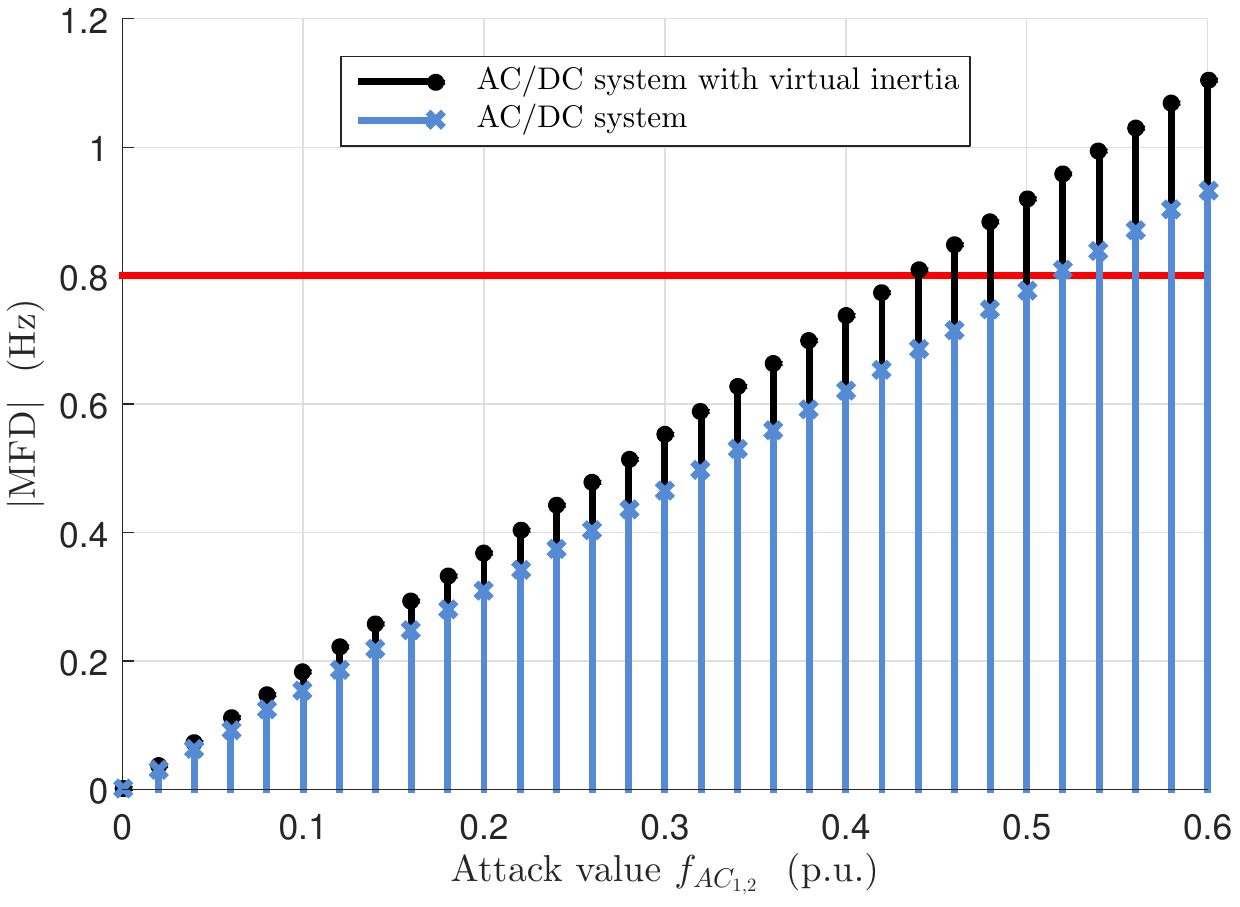}
	\caption{
		The MFD of Area 1 under various values of a univariate attack $f_{AC_{1,2}}$. 
		The red line indicates the MFD limit for having a disruptive FDI attack. }
	\label{fig:mfd_fac12}
\end{figure*}
\begin{Rem}[Vulnerability of different system models under FDI attacks]\label{rem:comp_threemodel}
	Consider the frequency dynamics models of the following systems under FDI attacks,
	\begin{itemize}
		\item normal AC system,
		\item AC/DC interconnected system, 
		\item AC/DC interconnected system with virtual inertia. 
	\end{itemize}
	We can see that, different from the pure AC system, the	hybrid AC/DC interconnected system with emulated inertia would have more vulnerable points to FDI attacks. Intruders can manipulate the wide-area measurements of frequencies and both AC and DC power flows. Besides, according to the high-level control structure, corruptions on these wide-area measurements can affect both the SPMC control and the secondary frequency control loops. In Section~\ref{subsec:result_vul} we would present such a detailed analysis. 
\end{Rem}

An advanced FDI attacker also considers the impact of various attack strategies. The frequency properties that can be influenced by the attacker are mainly maximum (center-of-inertia) frequency deviation (MFD) or steady-state frequency deviation (SSFD). In this article, we use the former MFD during the transients as the attack impact index. Intuitively, for a univariate FDI attack where only one measurement is compromised, a larger constant injection is more desired to cause the maximum damage. In this article, the attack is said to be disruptive to the system when the impact index MFD arrives at a certain value. For instance, in Figure~\ref{fig:mfd_fac12}, this (absolute) value is selected to be $0.8 \, \mathrm{Hz}$, as it may mislead to trigger a possible load shedding scheme when the frequency reaches $59.2 \, \mathrm{Hz}$ (for a $60 \, \mathrm{Hz}$ system) according to\cite{Entsoe2016}. In Figure~\ref{fig:mfd_fac12}, we can see that with the increase of the univariate attack value, the MFD becomes larger. To be disruptive, it can be inferred from Figure~\ref{fig:mfd_fac12} that the univariate attack value on the AC power flow measurement of the two-area AC/DC interconnected system with emulated inertia should be a minimum of $0.44 \, \mathrm{p.u.}$ (for the base of ${100}\ \mathrm{MW}$), while it is $0.52 \, \mathrm{p.u.}$ in the AC/DC system without virtual inertia. These values are obtained by simulations of the frequency dynamics model under the univariate attack $f_{AC_{1,2}}$. 

\subsection{Disruptive stealthy FDI attack strategies}\label{subsec:dis_steal_attack}

The univariate attack for a large impact in Figure~\ref{fig:mfd_fac12} may be detectable since it may go beyond possible thresholds and trigger data quality alarms. An intelligent attack should attend to pursue a desired impact and also satisfy the undetectability criterion \cite{Chen2018}. For that purpose, the attacker is required to have vast attack resources to manipulate multiple data channels (i.e., multivariate attacks) and enough knowledge of the targeted system operations the parameters of the frequency dynamics model in Section~\ref{sec:sys_model} and also the data quality checking programs). Then the attacker can select ``appropriate'' injection values. This ensures that the worst-case attack scenario is considered in the vulnerability analysis. The vulnerability of a multi-area system to such optimal attacks can be quantified by computing the attack resources needed by the attacker to achieve the targets on attack impact and undetectability, which is formalized in the following definition. %
\begin{Def}[Vulnerability to disruptive stealthy attacks] \label{def:vul_disrupt_stlattack}
	Consider an FDI attack with $\bm{f}$. We call it a {\em disruptive stealthy attack} if $\bm{f}$ takes values from the set
	\begin{equation}\label{eq:att_sp}
	\mathcal{F} := \Big\{ \bm{f} \in \mathbb{R}^{n_{Y}}: \ \bm{b_{min}} \leq \bm{F}_{f}\bm{f} \leq \bm{b_{max}} \Big\} \, ,
	\end{equation}
	where the vectors $\bm{b_{min}}, \, \bm{b_{max}} \in \mathbb{R}^{n_{b}}$ and the matrix $\bm{F_{f}} \in \mathbb{R}^{n_{b} \times n_{Y}}$ are scenario specific. The following Remark~\ref{rem:disrupt_stlattack_set} shows that the choice of $\bm{b_{min}}, \, \bm{b_{max}}, \, \bm{F_{f}}$ may be adjusted according to different national grid codes. 
	Thus, to describe the vulnerability of multi-area systems to the disruptive stealthy attacks, one can formulate an optimization program to compute the needed attack resources by the attacker to achieve its targets on attack impact and undetectability, 
	\begin{align}\label{opt:vul_disrupt_stlattack}
	\alpha_{i}^\star := \min\limits_{\bm{f}} \quad&  \| \, \bm{f} \, \|_{0} \nonumber \\
	\mbox{s.t.} \quad& \bm{f} \in \mathcal{F}, \ \ \bm{f}(i) = \mu, \\
	& \bm{f}(j) = 0, \ \ \mbox{for all } j \in \mathcal{P}. \nonumber
	\end{align}
	where $\|\cdot\|_{0}$ denotes the zero vector norm (the number of non-zero elements in a vector), $\bm{f}(i)$ is the injection with value $\mu$ on a specific measurement that the attacker can already access. We also add the constraint that some well-protected data channels (in the set $\mathcal{P}$) cannot be attacked. The multi-area system is more vulnerable to the attack with a smaller $\alpha_{i}^\star$ as it requires fewer data channels to be manipulated to achieve its targets on attack impact and undetectability. 
\end{Def}

The optimization program~\eqref{opt:vul_disrupt_stlattack} is NP-hard. We can use the big M method to express \eqref{opt:vul_disrupt_stlattack} as a mixed integer linear program (MILP) which can be solved by appropriate solvers. We omit the details and refer to \cite{TeixeiraSou2015} for a similar reformulation. 

\begin{Rem}[$\bm{b_{min}}, \, \bm{b_{max}}$ and $\bm{F_{f}}$ selection for disruptive stealthy attacks]\label{rem:disrupt_stlattack_set}
	For an effective attack strategy, the selection of parameters $\bm{b_{min}}, \, \bm{b_{max}}$ and $\bm{F_{f}}$ in \eqref{eq:att_sp} are critical. To be precise, the disruptive stealthy attacks need to satisfy the following criteria \cite{Sridhar2014,Ameli2018a},
	\begin{itemize}
		\item [(i)] To avoid triggering data quality alarms, the frequency deviation after attack corruptions should remain within a range, 
		\begin{align}\label{eq:dw_range}
		\Delta \omega_{min} \leq \Delta \omega_{i,f} \leq \Delta \omega_{max} \, , \nonumber
		\end{align}
		\item [(ii)] The calculated ACE of Area $i$ during attacks, $ACE_{i,f}$, should not exceed a permitted value, 
		\begin{equation}\label{eq:ace_max}
		| ACE_{i,f} | \leq ACE_{max} \, , \nonumber
		\end{equation}
		\item [(iii)] Similarly, the computed power reference signal for the HVDC link under FDI attack, $\Delta P_{DC_{ref,f}}$, should not exceed an acceptable value,
		\begin{align}\label{eq:dpdc_range}
		| \Delta P_{DC_{ref,f}} | \leq \Delta P_{DC_{ref, \, max}} \, , \nonumber
		\end{align}
		\item [(iv)] To be disruptive to the frequency stability, the MFD after FDI corruptions should reach a certain value. 
	\end{itemize}  	
\end{Rem}

It is worth mentioning that the limits of (i) - (iv) are system dependent as reflected in the different national grid codes. In this article, the values of $\Delta \omega_{min}$, $\Delta \omega_{max}$, $ACE_{max}$ and $\Delta P_{DC_{ref, \, max}}$ in (i) - (iii) are set to be $-0.1 \, \mathrm{Hz}$, $0.1 \, \mathrm{Hz}$, $0.05 \, \mathrm{p.u.}$ and $0.1 \, \mathrm{p.u.}$ respectively, according to the references \cite{kundur1994power,Hua2013b,Entsoe2016b,Garcia2017}. For (iv), as mentioned earlier, the (absolute) value of MFD should reach $0.8 \, \mathrm{Hz}$ for a disruptive FDI attack from the grid code in \cite{Entsoe2016}. %

\section{Attack Detection, Isolation and Recovery}\label{sec:detect}

In this section, a detector in the form of residual generator is developed for FDI attack detection, isolation and recovery. To do that, let us first reformulate the system frequency dynamics model in the state-space representation of Section~\ref{sec:sys_model} into a general DAE description. Consider a time-shift operator $q$ that $q\bm{x}[k] \rightarrow \bm{x}[k+1]$. One can fit \eqref{eq:spy_attacked} and \eqref{eq:spx_twoarea_attack} into,
\begin{equation}\label{eq:dae}
\bm{H}(q)\bm{x}[k] + \bm{L}(q)\bm{y}[k] + \bm{F}(q)\bm{f}[k] = 0,
\end{equation}
where $\bm{x}:= [\bm{X}^\top \ d^\top]^{\top}$ contains all the unknown signals of system states and ``disturbances'' (load variations of this article), $\bm{y} := \tilde{\bm{Y}}$ denotes the available system output for a detector. Let $n_{x}$, $n_{y}$ and $n_{r}$ be the dimensions of $\bm{x}[\cdot]$, $\bm{y}[\cdot]$ and the row number of \eqref{eq:dae}. Here $\bm{H}(\cdot)$, $\bm{L}(\cdot)$ and $\bm{F}(\cdot)$ are polynomial matrices in terms of $q$ such that
\begin{equation}\label{eq:dae_def}
\bm{H}(q) := \left[\begin{matrix} -q\bm{I} + \bm{A} & \bm{B}_{d}\\ \bm{C} & \bm{0}\\ \end{matrix}\right], \,  \bm{L}(q) := \left[\begin{matrix} \bm{0} \\ -\bm{I} \\ \end{matrix}\right], \,  \bm{F}(q) := \left[\begin{matrix} \bm{B}_{f}\\ \bm{I} \\ \end{matrix}\right]. \nonumber 
\end{equation}

\subsection{FDI attack detector construction}

The principle of an FDI attack detector is to generate a diagnosis signal to reveal the presence of the attack, giving the available data $\bm{y}[k]$. Definition~\ref{def:attack_detect} characterizes its task. 

\begin{Def}[FDI attack detection] \label{def:attack_detect}
	The diagnosis signal from the detector differentiates whether the system output is a consequence of 
	normal disturbances or FDI attacks. Thus ideally it relates a non-zero mapping from the attack to the diagnosis signal, while it is decoupled from the effects of unknown system states and disturbances. 
\end{Def} 

In this article, we restrict the attack detector to a type of residual generator with linear transfer operations \cite{Pan2020}, i.e., ${r}[k] := \bm{R}(q)\bm{y}[k]$, where ${r}[\cdot]$ is called the residual signal for diagnosis, $\bm{R}(q)$ is the transfer function that needs to be designed. Considering that $\bm{y}[\cdot]$ is associated with $\bm{L}(q)$ in \eqref{eq:dae}, we propose a formulation of $\bm{R}(q) := \bm{a}(q)^{-1}\bm{N}(q)\bm{L}(q)$. Now the task of detector construction comes to the design of $\bm{N}(q)$ whose dimension and predefined order are $n_{r}$ and $d_{N}$, if the denominator $\bm{a}(q)$ with sufficient order to make $\bm{R}(q)$ physically realizable is determined. Multiplying the left of \eqref{eq:dae} by $\bm{a}(q)^{-1}\bm{N}(q)$ would lead to 
\begin{equation}\label{eq:residual_gen}
\begin{aligned}
{r}[k] &= \bm{a}(q)^{-1}\bm{N}(q)\bm{L}(q)\bm{y}[k]  \\
&= -\underbrace{\bm{a}(q)^{-1}\bm{N}(q)\bm{H}(q)\bm{x}[k]}_{\text{(\RNum{1}})} - \underbrace{\bm{a}(q)^{-1}\bm{N}(q)\bm{F}(q)\bm{f}[k]}_{\text{(\RNum{2}})} \, .
\end{aligned}
\end{equation}
Term (I) is the part from the effect of unknown system states and disturbances $\bm{x}[\cdot]$. Term (II) corresponds to the FDI attack. Thus according to Definition~\ref{def:attack_detect}, the desired detector would generate the residual signal $r[\cdot]$ that can be decoupled from $\bm{x}[\cdot]$ but keep sensitive to $\bm{f}[\cdot]$. We would expect
\begin{align}\label{eq:desire_detect}
\bm{N}(q)\bm{H}(q) = 0 \, , \quad \bm{N}(q)\bm{F}(q) \neq 0 \, .
\end{align}
Inspired by \eqref{eq:desire_detect}, the following theoretical result shows an effective way for attack detection and also recovery. 

\begin{Thm}[FDI attack detection and recovery]\label{the:ss_uniatt_track}
	It can be observed that $\bm{H}(q) := \sum_{i = 0}^{1}\bm{H}_{i}q^{i}$, $\bm{N}(q) := \sum_{i = 0}^{d_{N}}\bm{N}_{i}q^{i}$ and $\bm{F}(q) := \bm{F}$. Consider an FDI attack in the set~\eqref{eq:att_sp}. A residual generator with the following linear program characterizations for \eqref{eq:desire_detect} can have non-zero steady-state residual output that recovers the attack value $\bm{f}$, 
	\begin{align}\label{eq:the_att_detect_recov}
	\left\{
	\begin{array}{ll} 
	& \bar{\bm{N}}\bar{\bm{H}} = 0 \, , \\ 
	& -\bm{a}(1)^{-1}\displaystyle\sum\limits_{i=0}^{d_{N}} \bm{N}_{i}\bm{F} = 1 \, ,
	\end{array}
	\right. 
	\end{align}
	where the matrices $\bar{\bm{N}}$, $\bar{\bm{H}}$ are defined as
	\begin{equation}\label{eq:NH_bar}
	\begin{aligned}
	\bar{\bm{N}} &:= \left[\begin{matrix} \bm{N}_{0} & \bm{N}_{1} &  \cdots & \bm{N}_{d_{N}} \end{matrix}\right], \nonumber\\
	\bar{\bm{H}} &:= \left[\begin{matrix} \bm{H}_{0} & \bm{H}_{1} & 0 & \cdots & 0 \\ 0 & \bm{H}_{0} & \bm{H}_{1} & 0 & \vdots \\ \vdots & 0  & \ddots & \ddots & 0 \\ 0 & \cdots & 0 & \bm{H}_{0} & \bm{H}_{1} \end{matrix}\right]. \nonumber
	\end{aligned}
	\end{equation}
\end{Thm}

\begin{proof}
	Note that $\bm{N}(q)\bm{H}(q) = \bar{\bm{N}}\bar{\bm{H}} [\bm{I} ,\ q\bm{I} ,\ \cdots ,\ q^{d_{N} + 1}\bm{I} ]^\top$. If $\bar{\bm{N}}\bar{\bm{H}}=0$ as stated in \eqref{eq:the_att_detect_recov}, then the diagnosis filter becomes $r[k] = -\bm{a}(q)^{-1}\bm{N}(q)\bm{f}[k]$. The steady-state value of the residual signal under the FDI attack would become $-\bm{a}(q)^{-1}\bm{N}(q)\bm{F}(q)\bm{f} \ | \ _{q=1}$. To track the FDI attack magnitude, one can simply make $-\bm{a}(1)^{-1}\bm{N}(1)\bm{F}(1) = 1$. Due to that $\bm{N}(1)\bm{F}(1) = \sum_{i=0}^{d_{N}} \bm{N}_{i}\bm{F}$, the residual signal from \eqref{eq:the_att_detect_recov} recovers the exact attack value $\bm{f}$ in the steady-state behavior of the residual generator. %
\end{proof}

Next, the residual generator design becomes to find a feasible $\bar{\bm{N}}$ satisfying~\eqref{eq:the_att_detect_recov}. To increase the sensitivity of the residual to the attack, in addition to \eqref{eq:the_att_detect_recov}, the detector may also aim to let the coefficients of $\bm{N}(q)\bm{F}(q)$ in \eqref{eq:desire_detect} achieve the maximum value. Then we can characterize the attack detection and recovery problem as the optimization program, 
\begin{align}\label{opt:detect}
\gamma^\star = \max\limits_{\bar{\bm{N}}} \quad&  \big\| \bar{\bm{N}}\bar{\bm{F}} \big \|_{\infty} \nonumber \\
\mbox{s.t.} \quad& \bar{\bm{N}}\bar{\bm{H}} = 0 \, , \quad \big\| \bar{\bm{N}} \big \|_{\infty} \leq \eta \, , \\
&  -\bm{a}(1)^{-1}\displaystyle\sum\limits_{i=0}^{d_{N}} \bm{N}_{i}\bm{F} = 1 \, , \nonumber
\end{align}
where the constraint $\| \bar{\bm{N}} \|_{\infty} \leq \eta$ is added to avoid possible unbounded solutions and it does not affect the performance of the detector. The matrix $\bar{\bm{F}}$ has a similar definition with $\bar{\bm{H}}$ in Theorem~\ref{the:ss_uniatt_track}. Thus if we have a resulted $\gamma^\star > 0$, then the solution to \eqref{opt:detect} offers a residual generator that detects the FDI attack and also tracks the attack value. 

Strictly speaking, the proposed optimization program~\eqref{opt:detect} is not a linear program (LP) due to the non-convex objective function. However, as explained by a similar argument in \cite[Lemma 4.3]{Esfahani2016}, one can view \eqref{opt:detect} as a family of $n$ standard LPs where $n$ is the number of columns of $\bar{\bm{F}}$.

\begin{Rem}[Attack isolation]\label{rem:attack_isolate}
	The residual generator from \eqref{opt:detect} is mainly designed for one univariate attack. For multivariate attacks ($\alpha_{i}^\star > 1$ from \eqref{opt:vul_disrupt_stlattack}), an alternative is to build a bank of residual generators where each of them aims to detect one particular intrusion and keep insensitive to others, by considering the following ``reconstructed'' DAE,
	\begin{equation}
	\Big[ \bm{H}(q) \ \bm{F}_{ - j}(q) \Big]  \left[\begin{matrix} {\bm{x}}[k] \\ \bm{f}_{-j}[k]\end{matrix}\right] + \bm{L}(q){\bm{y}}[k] + \bm{F}_{j}(q)\bm{f}_{j}[k] = 0, \nonumber
	\end{equation}
	where $\bm{F}_{ - j}(q)$ is the polynomial matrix that includes all columns of $\bm{F}(q)$ except the $j$-th one, and similarly $\bm{f}_{-j}[k]$ contains all the elements of $\bm{f}[k]$ except the $j$-th one. Then the $j$-th residual generator can be designed using the same approach in Theorem~\ref{the:ss_uniatt_track} and~\eqref{opt:detect} for the $j$-th intrusion while isolating the effects from others. The $j$-th attack can be identified by the $j$-th residual generator since the other residual generators keep insensitive to this attack. Besides, with~\eqref{opt:detect}, it can recover the $j$-th attack's value in the steady-state behavior of the $j$-th residual generator. 
\end{Rem}

\begin{figure}[t]
	\centering
	\includegraphics[scale=0.68]{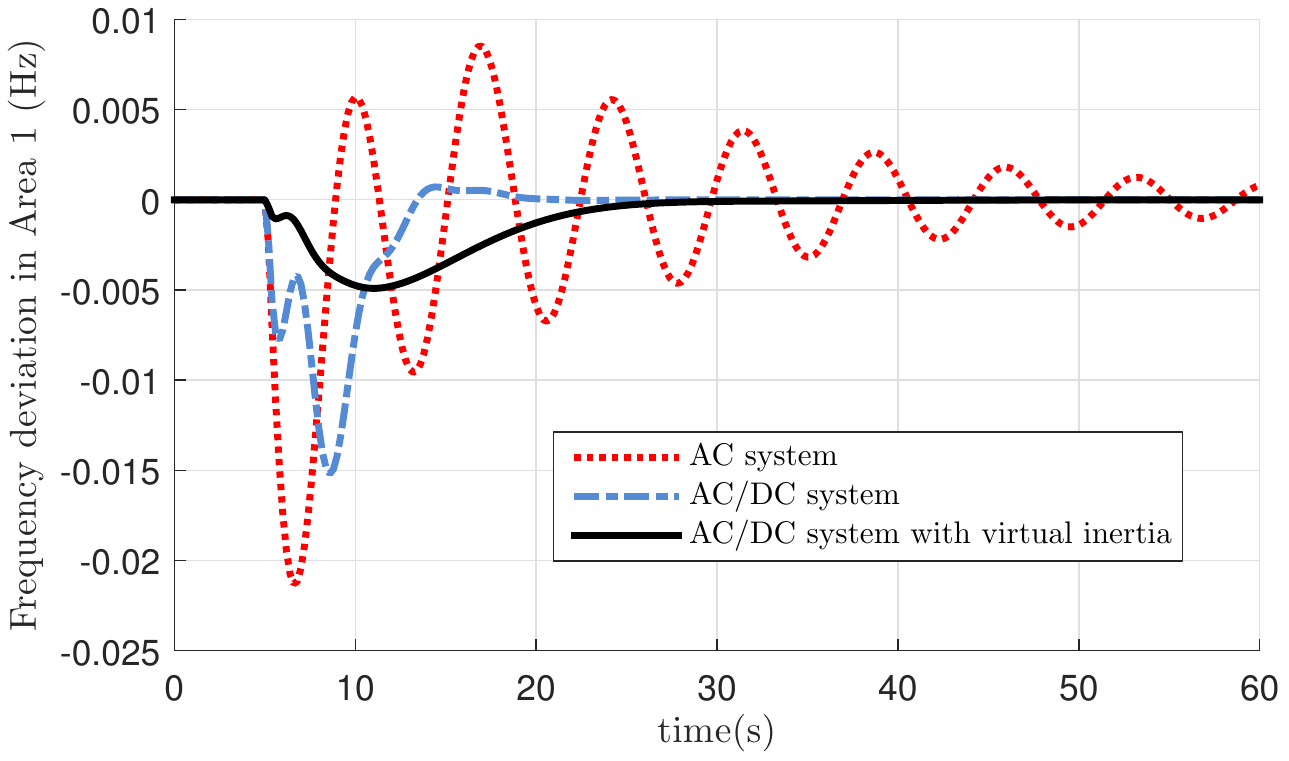}
	\caption{Frequency responses of Area 1 under a step load change of $ 3 \, \mathrm{\%}$ increase at $t = 5 \, \mathrm{s}$.} 
	\label{fig:freq1_d}
\end{figure}

Remark~\ref{rem:attack_isolate} shows that the attack isolation problem can be treated as an attack detection task effectively. In the end, we provide sufficient and necessary conditions for the feasibility of attack detection and isolation. %
\begin{Lem}[Necessary and sufficient conditions of attack detectability and isolability] \label{lem:detect_condition}
	For a univariate FDI attack ($\alpha_{i}^\star = 1$), it is detectable that satisfies \eqref{eq:desire_detect} if, and only if,
	\begin{align}\label{eq:detect_cond}
	Rank \ \big(\big[ \bm{H}(q) \ \, \bm{F}(q) \big] \big) > Rank \ \big(\bm{H}(q)\big). 
	\end{align}
	Besides, for multivariate FDI attacks ($\alpha_{i}^\star > 1$), one particular intrusion $\bm{f}_{j}$ is isolable from others if, and only if,
	\begin{align}\label{eq:isolate_cond}
	Rank \ \big( \big[ \bm{H}(q) \ \, \bm{F}(q)  \big] \big) > Rank \ \big( \big[ \bm{H}(q) \ \, \bm{F}_{-j}(q) \big] \big). 
	\end{align}
\end{Lem}
\begin{proof}
	The detectability condition \eqref{eq:detect_cond} is adapted from \cite[Theorem~3]{Nyberg2006a}. Alternatively, \eqref{eq:detect_cond} can be rewritten as $\bm{F}(q) \notin {\mbox{Im}} \big(\bm{H}(q)\big)$. It ensures that the residual signal keeps sensitive to the FDI attack but decoupled from the unknown system states and disturbances. For the isolability criterion, note that from Remark~\ref{rem:attack_isolate}, the $\bm{H}(q)$ in \eqref{eq:dae} has been ``replaced'' by $\big[ \bm{H}(q) \ \, \bm{F}_{ - i}(q) \big]$ in the reconstructed DAE. Then it is easy to obtain \eqref{eq:isolate_cond} extended from \eqref{eq:detect_cond}. Besides, the condition \eqref{eq:isolate_cond} can be also rewritten as $\bm{F}_{j}(q) \notin {\mbox{Im}} \big(\big[ \bm{H}(q) \ \, \bm{F}_{-j}(q) \big] \big)$. 
\end{proof}
\begin{figure*}[t!p]
	\centering
	\begin{subfigure}[t]{0.49\textwidth}
		\centering
		\includegraphics[scale=0.60]{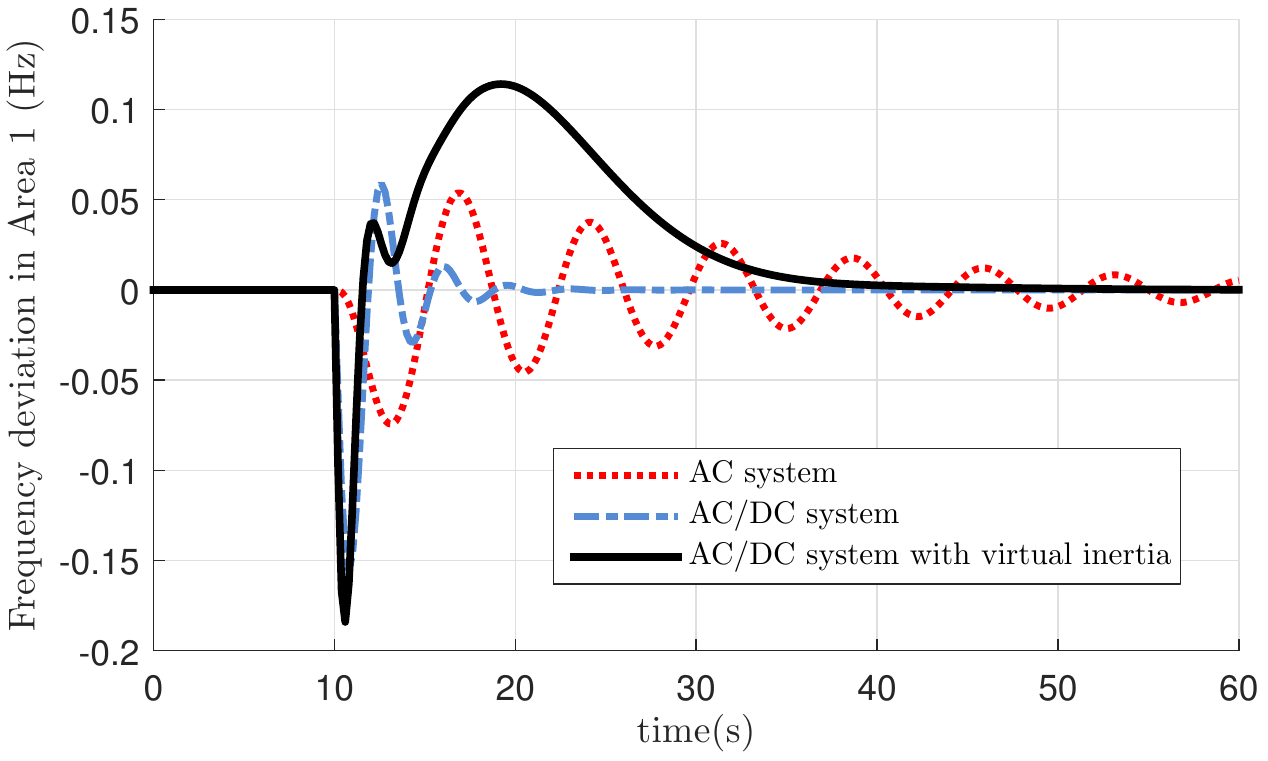}
		\caption{}\label{subfig:freq1_apac}
	\end{subfigure}
	~
	\begin{subfigure}[t]{0.49\textwidth}
		\centering
		\includegraphics[scale=0.60]{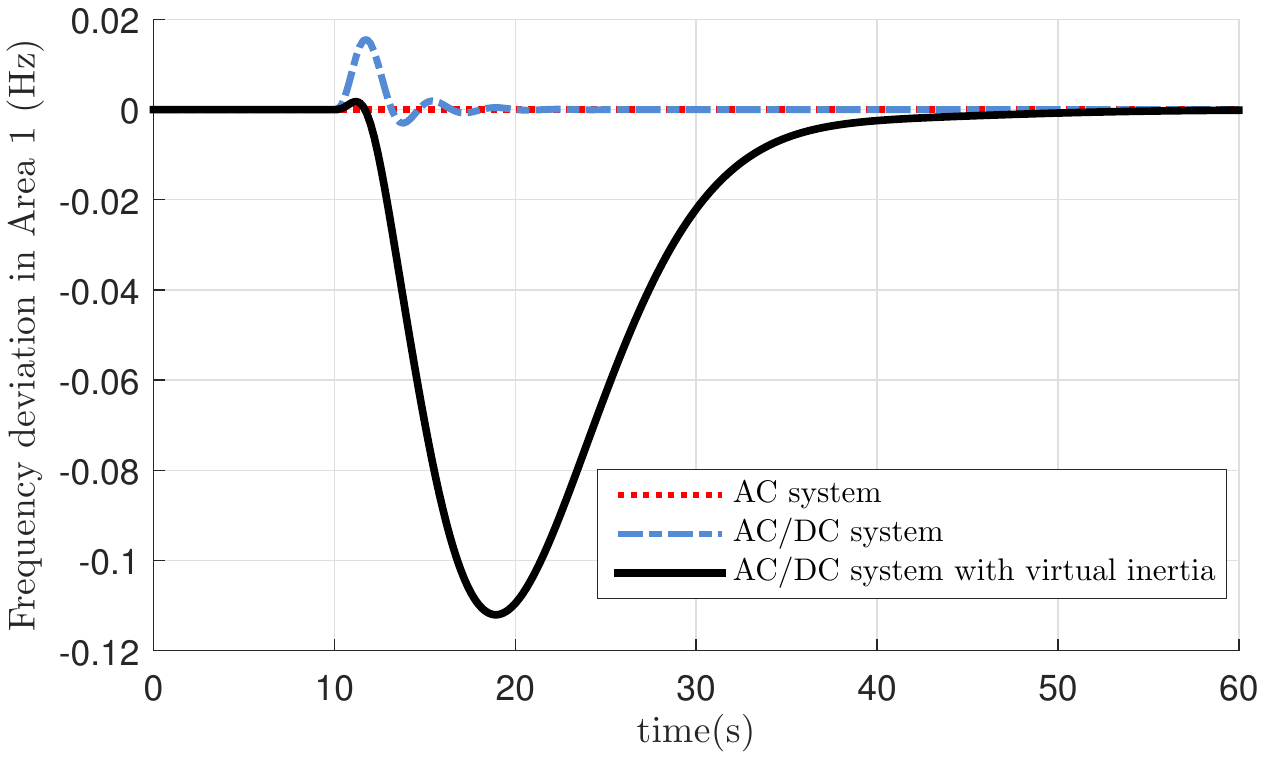}
		\caption{}\label{subfig:freq1_apdc}
	\end{subfigure}
	\\ 
	\begin{subfigure}[t]{0.49\textwidth}
		\centering
		\includegraphics[scale=0.60]{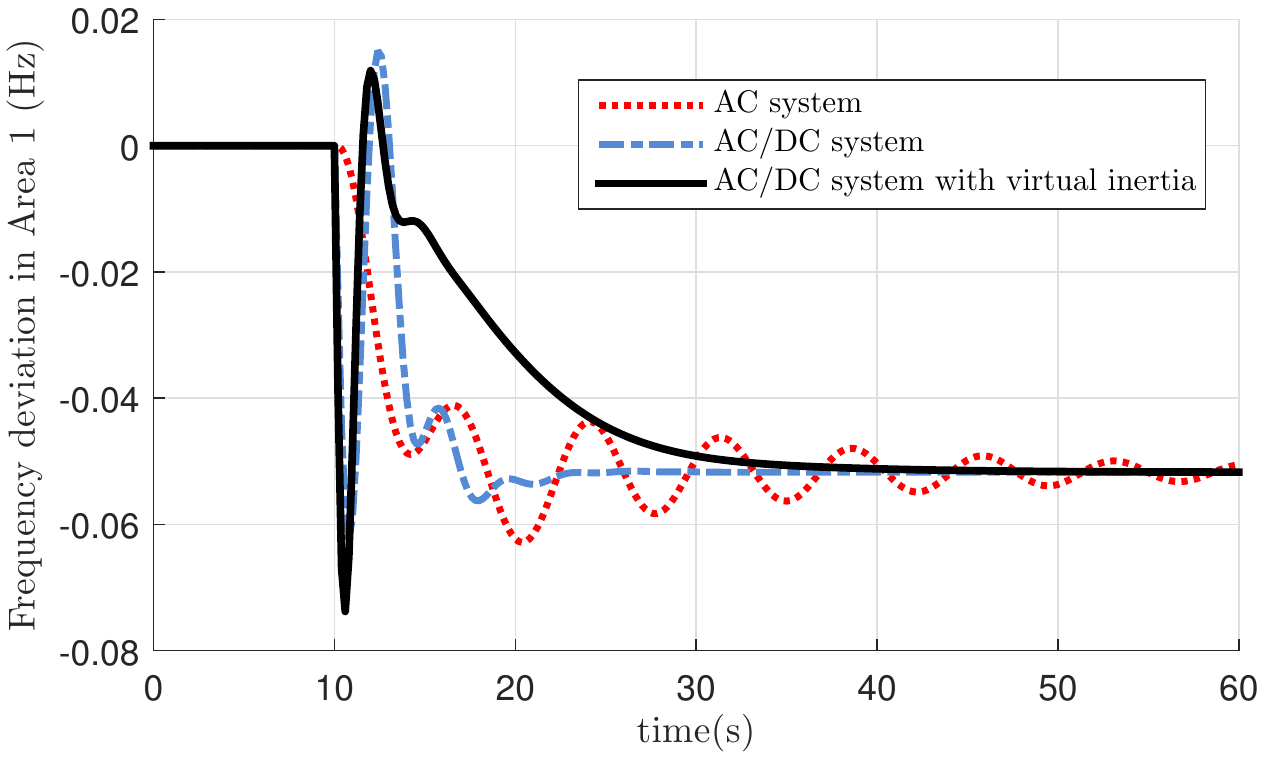}
		\caption{}\label{subfig:freq1_aw1}
	\end{subfigure}
	~	
	\begin{subfigure}[t]{0.49\textwidth}
		\centering
		\includegraphics[scale=0.60]{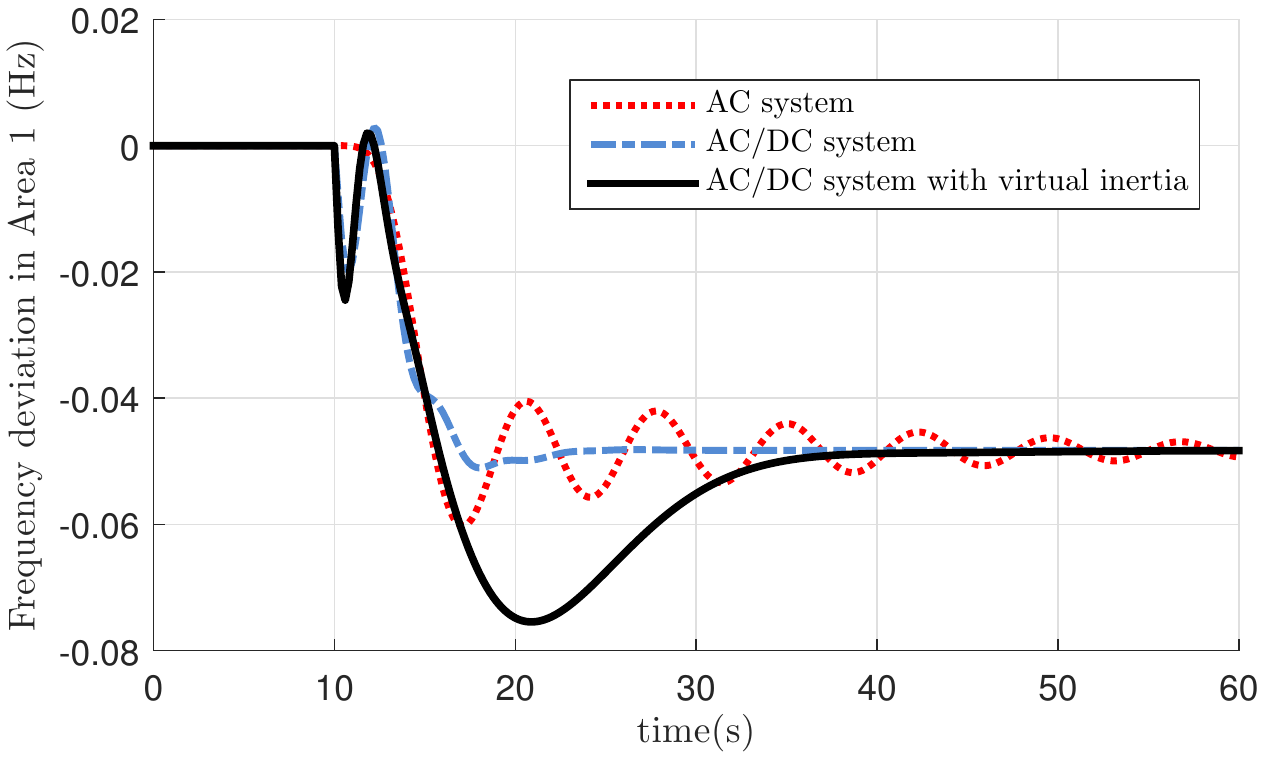}
		\caption{}\label{subfig:freq_aw2}
	\end{subfigure}
	\\ 
	\caption{Frequency responses of Area 1 under univariate attacks (a) on the AC link, $f_{AC_{1,2}} =0.1 \, \mathrm{p.u.}$; (b) on the DC link, $f_{DC_{1,2}} =0.1 \, \mathrm{p.u.}$; (c) on the frequency of Area 1, $f_{\omega_{1}} = 0.1 \, \mathrm{Hz}$; (d) on the frequency of Area 2, $f_{\omega_{2}} = 0.1 \, \mathrm{Hz}$ at $t = 10 \, \mathrm{s}$.} 
	\label{fig:freq1_a_allmeas}
\end{figure*}
\begin{figure}
	\centering
	\includegraphics[scale=0.60]{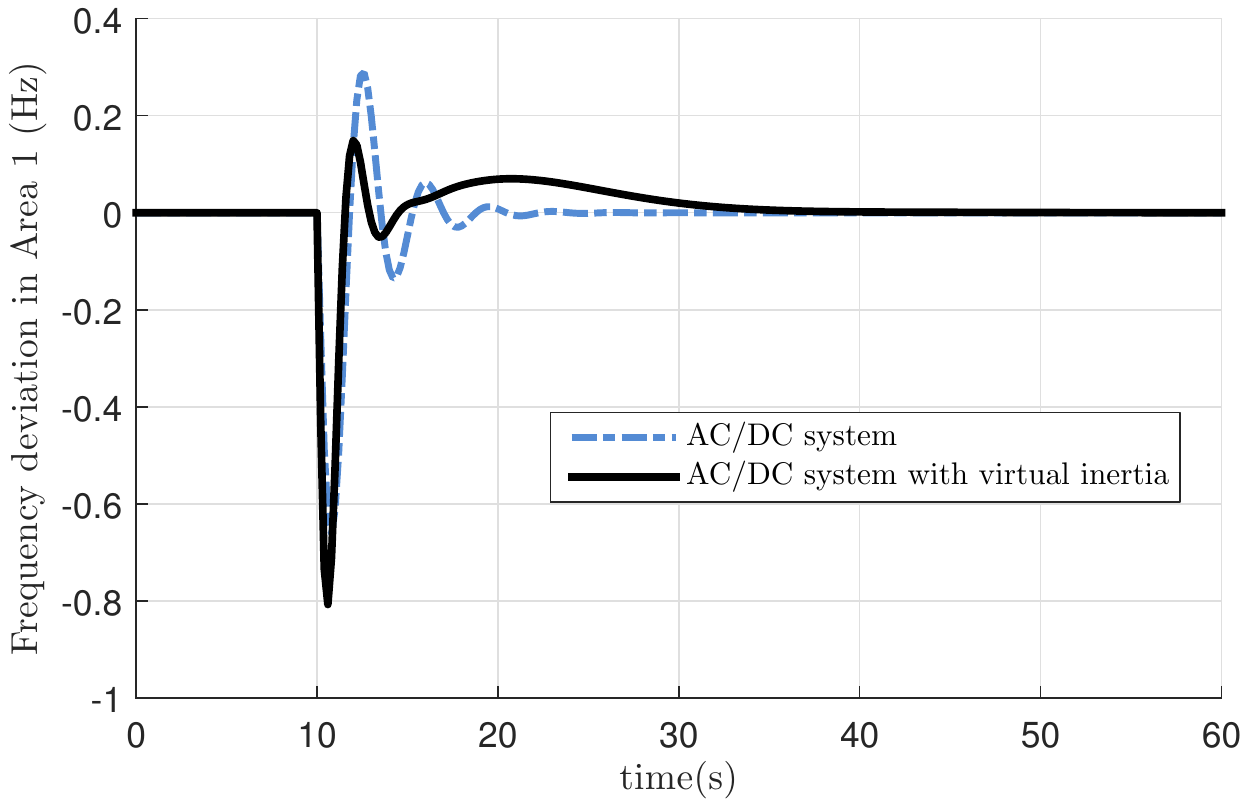}
	\caption{Frequency responses of Area 1 under the multivariate attacks 
		on both AC and DC links.}
	\label{fig:freq1_a_disruptstealth}
\end{figure}

\section{Numerical Results} \label{sec:result}

\subsection{Preliminaries} 
In order to evaluate the vulnerability and impact of the AC/DC interconnected system with virtual inertia to FDI attacks and also validate the effectiveness of the proposed attack detection methodology, in this section, we study the exemplary two-area system in Figure~\ref{fig:acdcsys} with parallel AC/DC links and ESS for inertia emulation and provide numerical results. As shown in Figure~\ref{fig:acdcsys}, there are two GENs and one load demand center in each area. The system parameters of the two-area system and the control parameters of this case study are referred 
\cite{Rakhshani2017a}. Given the characterizations of the time responses of controllers (especially inertia emulators) in the high-level structure, the sampling period $T_{s}$ is given as $0.04 \, \mathrm{s}$ \cite{Rakhshani2016b}. Till this end, we can obtain a 12-order discrete-time model which can be fitted into the DAE form of \eqref{eq:dae}.    

In the detector design, the degree of the residual generator is set to $d_{N} = 3$ which is much less than the order of the system. Besides, we give the denominator a determined formulation of $\bm{a}(q) := (q - p)^{d_{N}}/{{(1-p)}^{d_{N}}}$ where $p$ can be treated as the pole of $R(q)$, and it is normalized in steady-state value for all feasible poles. It should be mentioned that the parameters $d_{N}$ and $p$ are adjustable for a fast response of attack detection in the context of frequency dynamics considering inertia emulation. Particularly, as a general rule, the smaller the poles, the faster the residual responds, and the more sensitive to system noises \cite{Pan2020}. We use CPLEX to solve all the corresponding optimization programs.

\subsection{Vulnerability of two-area systems to FDI attacks} \label{subsec:result_vul}

First, we present the results of frequency deviations in Figure~\ref{fig:freq1_d} where the system input is only a step load change of $ 3 \, \mathrm{\%}$ increase at $t = 5 \, \mathrm{s}$. The comparisons are made on the normal AC system, the AC/DC interconnected system and the AC/DC interconnected system with virtual inertia. We can see that the HVDC link and especially, the inertia emulation provided by the ESS, can improve the system dynamics significantly in damping frequency oscillations during the step-load fault, which proves the effectiveness of the high-level frequency control structure of this article.

Next, to evaluate the vulnerability of different system models to FDI attacks, we launch univariate attacks on the wide-area measurements of frequencies and AC/DC power flows separately (recall Figure~\ref{fig:acdcsys} for vulnerable measurement uploading channels). The simulation results are shown in Figure~\ref{fig:freq1_a_allmeas}. In all of these scenarios, the univariate attack with the same attack value can cause the most severe ``damage'' to the AC/DC interconnected system with virtual inertia. Comparing with the normal AC system and the system with parallel AC/DC links but without emulated inertia, the AC/DC system with inertia emulators is always with the largest MFDs under each univariate attack scenario. This observation is consistent with Remark~\ref{rem:comp_threemodel} of Section~\ref{subsec:fdi_basic}. In fact, as stated in Remark~\ref{rem:comp_threemodel}, FDI corruptions on the wide-area frequencies and AC/DC power flows would affect both the SPMC control and the secondary frequency control loops as these measurements are inputs to these controllers in our high-level control structure. 

We continue with disruptive stealthy attacks introduced in Section~\ref{subsec:dis_steal_attack}. These FDI attacks with optimal attack strategies can be multivariate to achieve the targets on attack impact and undetectability, and are obtained by solving the optimization program \eqref{opt:vul_disrupt_stlattack} for vulnerability analysis. From the results of~\eqref{opt:vul_disrupt_stlattack}, an ``optimal'' multivariate attack ($\alpha_{i}^{\star} = 2$) that can manipulate both AC and DC power lines with $f_{AC_{1,2}} = 0.44 \, \mathrm{p.u.}$ and $f_{DC_{1,2}} = -0.39 \, \mathrm{p.u.}$ is a disruptive stealthy attack in the set of \eqref{eq:att_sp} for the two-area AC/DC interconnected system with virtual inertia. Figure~\ref{fig:freq1_a_disruptstealth} shows the frequency response of Area 1 under this multivariate attack. The MFD reaches $-0.8 \, \mathrm{Hz}$ at around $t = 10.6 \, \mathrm{s}$ while the attack is launched at $t = 10 \, \mathrm{s}$, which implies a disruptive attack as defined in this article. The MFD of the AC/DC system without virtual inertia has reached $-0.66 \, \mathrm{Hz}$ under the same multivariate attack. To be noted, there is no disruptive stealthy attack in the type of Definition~\ref{def:vul_disrupt_stlattack} when solving \eqref{opt:vul_disrupt_stlattack} for the normal AC system. Thus it is reasonable to conclude 
that the AC/DC system with virtual inertia is more vulnerable to FDI attacks.

\subsection{FDI attack detection and recovery} \label{subsec:result_detect}

In the third simulation, we validate the proposed methodology of attack detection, isolation and recovery. To challenge the detector, the system input $\bm{d}$ is modeled as stochastic load patterns; see Figure~\ref{subfig:stoload} of $\Delta P_{L_{1}}$. The adversarial cases come from the disruptive stealthy attacks obtained in the above section. We build a bank of two residual generators to detect and isolate the multivariate attacks $f_{AC_{1,2}}$ and $f_{DC_{1,2}}$ on the AC/DC links in the two-area system with virtual inertia, using the approach in \eqref{opt:detect} and Remark~\ref{rem:attack_isolate}. The optimal values of~\eqref{opt:detect} achieve $\gamma^{\star} = 4.440$ in the residual generator construction for detecting $f_{AC_{1,2}}$ and $\gamma^{\star} = 2.649$ in the other residual generator for $f_{DC_{1,2}}$, which implies a successful detection and isolation as indicated by Lemma~\ref{lem:detect_condition}. 

In Figure~\ref{fig:d_att_n_residuals}, the results of the residual generators under disturbances and attacks are presented. Both detectors have generated a residual signal for the presence of each FDI intrusion in the multivariate attack scenario. Besides, we can see that the resulted residual generators with designed capabilities from Theorem~\ref{the:ss_uniatt_track} can recover the attack values; the steady-state residual values in Figure~\ref{subfig:residual_apac} and Figure~\ref{subfig:residual_apdc} are equal to $ 0.44 \, \mathrm{p.u.}$ and $-0.39 \, \mathrm{p.u.}$. The residual signals are also decoupled from each other and stochastic load variations. Next, to test
the residual generators in a more realistic setting, we also provide the simulation results where there exist noises in the system process and measurements. Zero-mean Gaussian noises with the covariance 0.0009 to the frequency and 0.03 to the other states are applied \cite{Ameli2018a}. Figure~\ref{subfig:residual_apac_n} shows one instance of the residual signals and it still works effectively in detecting and tracking the attack value of $f_{AC_{1,2}}$. In the end, it should be noted that in these simulations the adjustable parameter, i.e., the pole of $a(q)$, is set to be $p=\, 0.1$ for a fast response of attack detection. The residual signal under $f_{AC_{1,2}}$ can recover the attack value from $t = 10.36 \, \mathrm{s}$ in Figure~\ref{subfig:residual_apac} before the MFD reaches the maximum value (at around $t = 10.6 \, \mathrm{s}$) in Figure~\ref{fig:freq1_a_disruptstealth}. %
This indicates that the developed residual generator can detect the FDI intrusions sufficiently fast in the inertia context. %

\begin{figure}[t!p]
	\begin{subfigure}[t]{0.49\textwidth}
		\centering
		\includegraphics[scale=0.65]{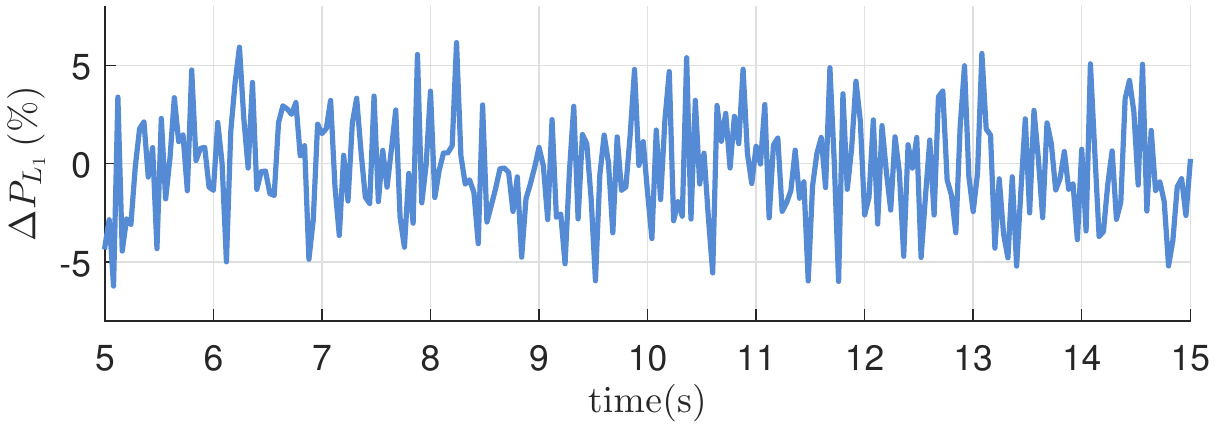}
		\caption{Load variations of Load 1 in Area 1.}\label{subfig:stoload}
	\end{subfigure}
	\\ 
	\begin{subfigure}[t]{0.49\textwidth}
		\centering
		\includegraphics[scale=0.65]{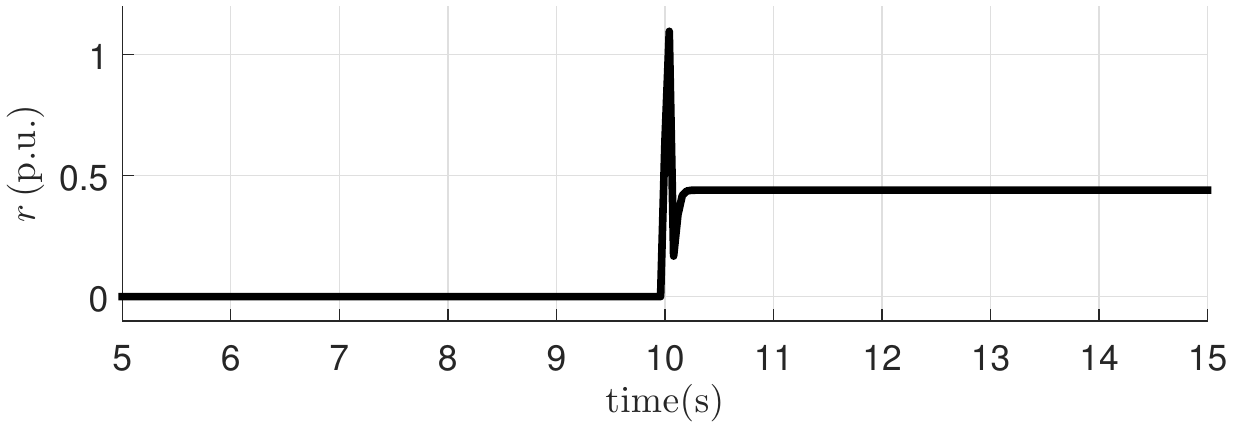}
		\caption{Residual signal under $f_{AC_{1,2}}$}\label{subfig:residual_apac}
	\end{subfigure}
	\\
	\begin{subfigure}[t]{0.49\textwidth}
		\centering
		\includegraphics[scale=0.65]{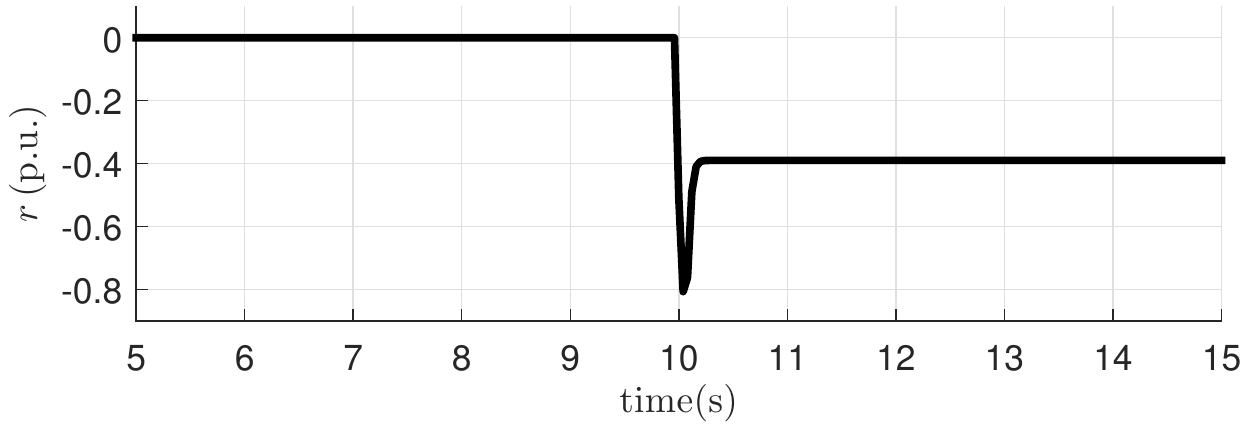}
		\caption{Residual signal under $f_{DC_{1,2}}$}\label{subfig:residual_apdc}
	\end{subfigure}
	\\
	\begin{subfigure}[t]{0.49\textwidth}
		\centering
		\includegraphics[scale=0.65]{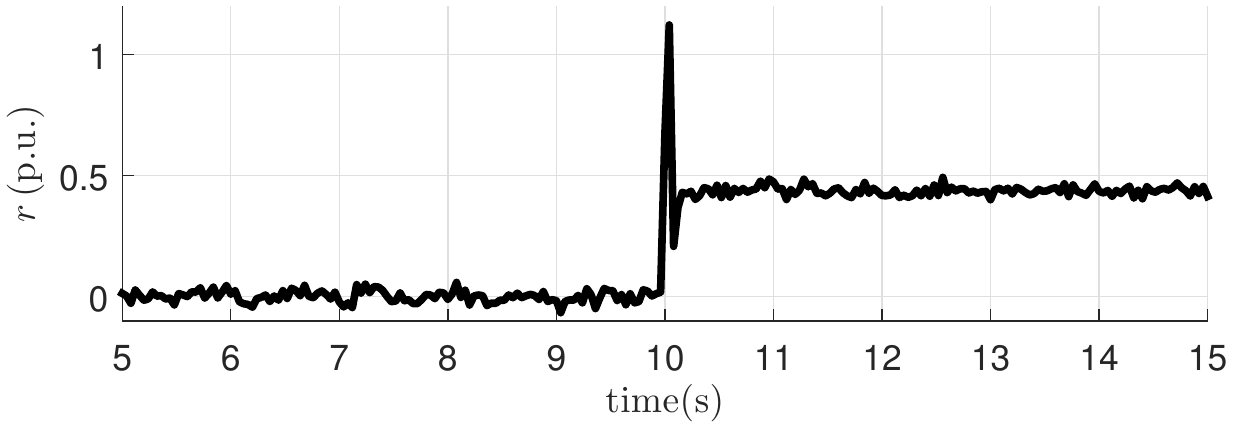}
		\caption{Residual signal under $f_{AC_{1,2}}$ and noises.}\label{subfig:residual_apac_n}
	\end{subfigure}
	\\
	\caption{Residual responses under multivariate attacks.}\label{fig:d_att_n_residuals} 
\end{figure}
%

%

\section{Conclusion} \label{sec:conclusion}

In this article, we investigated FDI attacks on the hybrid AC/DC system with virtual inertia. We offer an optimization-based vulnerability and attack impact analysis framework. Our study shows that the hybrid grid can be more vulnerable to FDI attacks. We also propose a diagnosis tool to detect, isolate and recover all the FDI intrusions. The effectiveness of these methods was validated by simulations in the two-area AC/DC system with emulated inertia. The future research includes the study of other types of cyber attacks on the hybrid grids with virtual inertia emulation capabilities. 

\bibliographystyle{IEEEtran}
\bibliography{literature}

\begin{thebibliography}{10}
\providecommand{\url}[1]{#1}
\csname url@samestyle\endcsname
\providecommand{\newblock}{\relax}
\providecommand{\bibinfo}[2]{#2}
\providecommand{\BIBentrySTDinterwordspacing}{\spaceskip=0pt\relax}
\providecommand{\BIBentryALTinterwordstretchfactor}{4}
\providecommand{\BIBentryALTinterwordspacing}{\spaceskip=\fontdimen2\font plus
\BIBentryALTinterwordstretchfactor\fontdimen3\font minus
  \fontdimen4\font\relax}
\providecommand{\BIBforeignlanguage}[2]{{%
\expandafter\ifx\csname l@#1\endcsname\relax
\typeout{** WARNING: IEEEtran.bst: No hyphenation pattern has been}%
\typeout{** loaded for the language `#1'. Using the pattern for}%
\typeout{** the default language instead.}%
\else
\language=\csname l@#1\endcsname
\fi
#2}}
\providecommand{\BIBdecl}{\relax}
\BIBdecl

\bibitem{Rakhshani2016a}
E.~Rakhshani, D.~Remon, and P.~Rodriguez, ``Effects of pll and frequency
  measurements on lfc problem in multi-area hvdc interconnected systems,''
  \emph{International Journal of Electrical Power \& Energy Systems}, vol.~81,
  pp. 140--152, 2016.

\bibitem{Xu2018}
T.~{Xu}, W.~{Jang}, and T.~{Overbye}, ``Commitment of fast-responding storage
  devices to mimic inertia for the enhancement of primary frequency response,''
  \emph{IEEE Transactions on Power Systems}, vol.~33, no.~2, pp. 1219--1230,
  Mar. 2018.

\bibitem{Dhingra2018}
K.~{Dhingra} and M.~{Singh}, ``Frequency support in a micro-grid using virtual
  synchronous generator based charging station,'' \emph{IET Renewable Power
  Generation}, vol.~12, no.~9, pp. 1034--1044, 2018.

\bibitem{Mosca2019}
C.~{Mosca}, F.~{Arrigo}, A.~{Mazza}, E.~{Bompard}, E.~{Carpaneto}, G.~{Chicco},
  and P.~{Cuccia}, ``Mitigation of frequency stability issues in low inertia
  power systems using synchronous compensators and battery energy storage
  systems,'' \emph{IET Generation, Transmission \& Distribution}, vol.~13,
  no.~17, pp. 3951--3959, 2019.

\bibitem{Liang2017a}
G.~Liang, S.~R. Weller, J.~Zhao, F.~Luo, and Z.~Y. Dong, ``The 2015 {Ukraine}
  blackout: Implications for false data injection attacks,'' \emph{IEEE Trans.
  on Power Syst.}, vol.~32, no.~4, pp. 3317--3318, Jul. 2017.

\bibitem{Sridhar2010}
S.~Sridhar and G.~Manimaran, ``Data integrity attacks and their impacts on
  scada control system,'' in \emph{Proc. IEEE PES General Meeting}, Jul. 2010,
  pp. 1--6.

\bibitem{Tan2017}
R.~Tan, H.~H. Nguyen, E.~Y.~S. Foo, D.~K.~Y. Yau, Z.~Kalbarczyk, R.~K. Iyer,
  and H.~B. Gooi, ``Modeling and mitigating impact of false data injection
  attacks on automatic generation control,'' \emph{IEEE Transactions on
  Information Forensics and Security}, vol.~12, no.~7, pp. 1609--1624, Jul.
  2017.

\bibitem{Fang2019}
J.~{Fang}, R.~{Zhang}, H.~{Li}, and Y.~{Tang}, ``Frequency derivative-based
  inertia enhancement by grid-connected power converters with a
  frequency-locked-loop,'' \emph{IEEE Transactions on Smart Grid}, vol.~10,
  no.~5, pp. 4918--4927, Sep. 2019.

\bibitem{Ruttledge2015}
L.~Ruttledge and D.~Flynn, ``Short-term frequency response of power systems
  with high non-synchronous penetration levels,'' \emph{Wiley Interdisciplinary
  Reviews: Energy and Environment}, pp. 452--470, 2015.

\bibitem{Datta2013}
M.~{Datta} and T.~{Senjyu}, ``Fuzzy control of distributed {PV}
  inverters/energy storage systems/electric vehicles for frequency regulation
  in a large power system,'' \emph{IEEE Transactions on Smart Grid}, vol.~4,
  no.~1, pp. 479--488, Mar. 2013.

\bibitem{Dudurych2017}
I.~Dudurych, M.~Burke, L.~Fisher, M.~Eager, and K.~Kelly, ``Operational
  security challenges and tools for a synchronous power system with high
  penetration of non-conventional sources,'' \emph{CIGRE Science \&
  Engineering}, no.~7, pp. 91--101, 2017.

\bibitem{TeixeiraSou2015}
A.~Teixeira, K.~C. Sou, H.~Sandberg, and K.~H. Johansson, ``Secure control
  systems: A quantitative risk management approach,'' \emph{IEEE Control
  Systems}, vol.~35, no.~1, pp. 24--45, 2015.

\bibitem{Esfahani2010}
P.~M. Esfahani, M.~Vrakopoulou, K.~Margellos, J.~Lygeros, and G.~Andersson,
  ``Cyber attack in a two-area power system: Impact identification using
  reachability,'' in \emph{American Control Conf.}, Jun. 2010, pp. 962--967.

\bibitem{Nyberg2006a}
M.~Nyberg and E.~Frisk, ``Residual generation for fault diagnosis of systems
  described by linear differential-algebraic equations,'' \emph{IEEE Trans. on
  Autom. Control}, vol.~51, no.~12, pp. 1995--2000, 2006.

\bibitem{Ameli2018a}
A.~Ameli, A.~Hooshyar, E.~El-Saadany, and A.~Youssef, ``Attack detection and
  identification for automatic generation control systems,'' \emph{IEEE
  Transactions on Power Systems}, p.~1, 2018.

\bibitem{Pan2020}
K.~{Pan}, P.~{Palensky}, and P.~M. {Esfahani}, ``From static to dynamic anomaly
  detection with application to power system cyber security,'' \emph{IEEE
  Transactions on Power Systems}, pp. 1584--1596, 2020.

\bibitem{Sridhar2014}
S.~Sridhar and M.~Govindarasu, ``Model-based attack detection and mitigation
  for automatic generation control,'' \emph{IEEE Transactions on Smart Grid},
  vol.~5, no.~2, pp. 580--591, Mar. 2014.

\bibitem{Pan2018}
K.~Pan, A.~Teixeira, M.~Cvetkovic, and P.~Palensky, ``Cyber risk analysis of
  combined data attacks against power system state estimation,'' \emph{IEEE
  Transactions on Smart Grid}, p.~1, 2018.

\bibitem{Amir2019}
A.~Gholami, M.~Mousavi, A.~K. Srivastava, and A.~Mehrizi-Sani, ``Cyber-physical
  vulnerability and security analysis of power grid with hvdc line,'' in
  \emph{2019 North American Power Symposium (NAPS)}, Wichita, Kansas, 2019, pp.
  1--6.

\bibitem{Brown2018}
H.~E. {Brown} and C.~L. {Demarco}, ``Risk of cyber-physical attack via load
  with emulated inertia control,'' \emph{IEEE Transactions on Smart Grid},
  vol.~9, no.~6, pp. 5854--5866, 2018.

\bibitem{Roy2019}
S.~D. {Roy} and S.~{Debbarma}, ``Detection and mitigation of cyber-attacks on
  {AGC} systems of low inertia power grid,'' \emph{IEEE Systems Journal}, pp.
  1--9, 2019.

\bibitem{Rakhshani2016b}
E.~{Rakhshani}, D.~{Remon}, A.~{Mir Cantarellas}, and P.~{Rodriguez},
  ``Analysis of derivative control based virtual inertia in multi-area
  high-voltage direct current interconnected power systems,''
  \emph{Transmission Distribution IET Generation}, vol.~10, no.~6, pp.
  1458--1469, 2016.

\bibitem{Pan2017a}
K.~Pan, A.~Teixeira, C.~D. L{\'o}pez, and P.~Palensky, ``Co-simulation for
  cyber security analysis: Data attacks against energy management system,'' in
  \emph{Smart Grid Communications (SmartGridComm), 2017 IEEE International
  Conference on}.\hskip 1em plus 0.5em minus 0.4em\relax IEEE, 2017, pp.
  253--258.

\bibitem{Rakhshani2017a}
E.~{Rakhshani} and P.~{Rodriguez}, ``Inertia emulation in {AC}/{DC}
  interconnected power systems using derivative technique considering frequency
  measurement effects,'' \emph{IEEE Transactions on Power Systems}, vol.~32,
  no.~5, pp. 3338--3351, Sep. 2017.

\bibitem{Ogata1995}
K.~Ogata, \emph{Discrete-time Control Systems (2Nd Ed.)}.\hskip 1em plus 0.5em
  minus 0.4em\relax Upper Saddle River, NJ, USA: Prentice-Hall, Inc., 1995.

\bibitem{Entsoe2016}
ENTSO-E, ``Continental europe operation handbook, policy 1: Load-frequency
  control and performance – appendix,'' ENTSO-E, Brussels, Belgium, Tech.
  Rep., Mar. 2016.

\bibitem{Chen2018}
C.~Chen, K.~Zhang, K.~Yuan, L.~Zhu, and M.~Qian, ``Novel detection scheme
  design considering cyber attacks on load frequency control,'' \emph{IEEE
  Transactions on Industrial Informatics}, vol.~14, no.~5, pp. 1932--1941,
  2018.

\bibitem{kundur1994power}
\BIBentryALTinterwordspacing
P.~Kundur, N.~Balu, and M.~Lauby, \emph{Power System Stability and Control},
  ser. Discussion Paper Series.\hskip 1em plus 0.5em minus 0.4em\relax
  McGraw-Hill Education, 1994. [Online]. Available:
  \url{https://books.google.nl/books?id=2cbvyf8Ly4AC}
\BIBentrySTDinterwordspacing

\bibitem{Hua2013b}
H.~Weng and Z.~Xu, ``Wams based robust hvdc control considering model
  imprecision for ac/dc power systems using sliding mode control,''
  \emph{Electric Power Systems Research}, vol.~95, pp. 38--46, 2013.

\bibitem{Entsoe2016b}
ENTSO-E, ``Frequency stability evaluation criteria for the synchronous zone of
  continental europe,'' ENTSO-E, Tech. Rep., 2016.

\bibitem{Garcia2017}
J.~Garcia, K.~Mudunkotuwa, and C.~Shumski, ``Migrate project – type-3 and
  type-4 emt model documentation,'' Manitoba HVDC Research Centre, Winnipeg,
  Tech. Rep., 2017.

\bibitem{Esfahani2016}
P.~M. Esfahani and J.~Lygeros, ``A tractable fault detection and isolation
  approach for nonlinear systems with probabilistic performance,'' \emph{IEEE
  Trans. on Autom. Control}, vol.~61, no.~3, pp. 633--647, Mar. 2016.

\end{thebibliography}

%
%
%

\end{document}